\documentclass[12pt]{amsart}
\usepackage{amssymb}
\usepackage{palatino}
\input amssym.def

\usepackage{longtable}
\usepackage{tcolorbox}

\allowdisplaybreaks

\usepackage[colorinlistoftodos]{todonotes}

\usepackage{amsmath, amsfonts}
\usepackage{amssymb}
\usepackage{amscd}
\usepackage[mathscr]{eucal}
\usepackage{palatino}
\usepackage{tabularx}
\usepackage{graphicx}
\usepackage{adjustbox}
\newfont{\cyrr}{wncyr10}

\newtheorem{thm}{Theorem}
\newtheorem{lem}{Lemma}
\newtheorem{cor}{Corollary}

\newcommand{\thmref}[1]{Theorem~\ref{#1}}

\newcommand{\lemref}[1]{Lemma~\ref{#1}}

\newcommand{\corref}[1]{Corollary~\ref{#1}}

\newcommand{\mbb}{\mathbb}

\newcommand{\psmb}{\left( \begin{smallmatrix}}
\newcommand{\psme}{ \end{smallmatrix} \right) }

\usepackage[utf8]{inputenc}

\thanks{\textit{Acknowledgment:}  The first author is supported by INSA Senior Scientist award.  The second author is supported by NBHM postdoctoral fellowship with grant number 0204/19/2017/R\& D-II/10383 at IISc Bangalore. }

\begin{document}

\title[Extension of Laguerre polynomials with negative arguments II]{Extension of Laguerre polynomials with negative arguments II}

\author{T. N. Shorey and Sneh Bala Sinha}

\address{T. N. Shorey \hfill Sneh Bala Sinha \newline
NIAS, Bangalore,   \hfill IISc Bangalore, \newline
560012, India.  \hfill   Department of Mathematics, \newline
$\phantom{mmmmmmmmmm}$ \hfill 560012, India. }

\email{shorey@math.iitb.ac.in} \hfill \email{24.sneh@gmail.com  ;
snehasinha@iisc.ac.in} 




\maketitle  

\begin{abstract}
For $n \geq 3$ and $s \leq 92$, it is proved in \cite{ShSi} that, except for finitely many pairs $(n, s), G_1(x) = G_1(x, n, s) $ is either irreducible or linear factor times an irreducible polynomial. If $s \leq 30$, we determine here explicitely the set of pairs $(n, s)$ in the above assertion. This implies a new proof of the result of Nair and Shorey \cite{NaSh1} that $G_1(x)$ is irreducible for $s \leq 22$.
\end{abstract} 

\section{\bf Introduction}

This is a continuation of \cite{ShSi}. Therefore we shall follow the notations of \cite{ShSi} but we shall recall here the key notations and key results from \cite{ShSi}. The generalised Laguerre polynomial of degree $n$ with negative argument is
$$
L_n^{(\alpha)} (x) = \sum\limits_{j=0}^{n} \frac{(\alpha+n) \dots (\alpha+j+1)}{ (n-j)!}\frac{(-x)^j}{j!}
$$
where $\alpha $ is negative. Then for $\alpha = -n-s-1$ where $s$ is a non-negative integer, we have
\begin{align*}
g(x) := g(x, n, s) = (-1)^n L_n^{(-n-s-1)}(x) & = \sum\limits_{j=0}^{n}  a_j \frac{x^j}{j!} \label{new_eq_1}  
\end{align*}
where $ a_j= {n+s -j \choose n-j }$ for $ 0 \leq j \leq n$.

Thus $a_n = 1$ and $  a_0 = {n+s \choose n} = \frac{(n+1) \dots (n+s)}{s!}$ and
\begin{equation*}
G(x) :=G(x,n,s )  =  \sum\limits_{j=0}^{n}  \pi_j \frac{x^j}{j!}  \qquad \text{ where} \quad \pi_j = b_ja_j
\end{equation*}
such that $b_j \in \mbb Z$ for $0 \leq j \leq n$ with $|b_0| =1$, $|b_n| = 1$. For $k \geq 1$ we say we have $(n, k, s)$ if $G(x) = G( x, n, s)$ has a factor of degree $k$ and we do not have $(n, k, s)$ if $G(x)$ has no factor of degree $k$. Next we write 
$$
g_1(x) = n! g(x), \quad G_1(x) = n! G(x). 
$$
Schur proved that  $G_1(x) $ with $s=0$ is irreducible. Therefore we always assume that $s > 0$.

\section{Lemmas}

\begin{lem}\label{FFL}
Let $k$ and $l$ be integers with $ k > l \geq 0$. Suppose that $h (x) = \sum\limits_{j=0}^{n} b_j x^j$ and $p$ prime such that $ p \nmid b_n$ and $p \mid b_j$ for $0 \leq j < n-l$ and the right most edge of the Newton polygon for $h(x)$ with respect to $p$ has slope less than $\frac{1}{k}$. Then for any $a_0, a_1, \dots, a_n \in \mbb Z$ with $|a_0| = |a_n| = 1$, the polynomial $f(x) = \sum\limits_{j=0}^{n} a_j b_j x^j \in \mbb Z[x]$ can not have a factor with degree in the interval $[l+1, k]$.
\end{lem} 

The next result is Lemma 1 from \cite{ShSi}.

\begin{lem}\label{lem_k_eq_1}
Assume that $G_1(x) $ has a factor of degree $1$. Then 
$$
n \leq s^{\pi(s)}.
$$
\end{lem}

Further we state the following result from \cite{ShSi}.

\begin{lem}\label{lem2k_le_n_le_4k}
Let $ n \geq 3$. Assume that $G_1(x)$ has a factor of degree $k \geq 2$.  Then $ s > 92$ unless
\begin{align*}
 (n,k,s) \in  & \{ (4,2,7),(4,2,23), ( 9, 2, 19), ( 9, 2, 47), (16, 2, 14), ( 16, 2, 34), ( 16, 2, 89),  (9,3,47), \\ & ( 16, 3, 19), (10,5,4) \nonumber\}.
\end{align*}
\end{lem}

As an immediate consequence of \lemref{lem2k_le_n_le_4k}, we derive the following result.

\begin{lem}\label{lem_s_92}
Let $n \geq 3$ and $s \leq 92$. Except for finitely many triples
\begin{align*}
 (n,k,s) \in  & \{ (4,2,7),(4,2,23), ( 9, 2, 19), ( 9, 2, 47), (16, 2, 14), ( 16, 2, 34), ( 16, 2, 89),  (9,3,47), \\ & ( 16, 3, 19), (10,5,4) \nonumber\},
\end{align*}
$G_1(x)$ is either irreducible or 
\begin{equation}\label{eq_new_G1}
G_1(x) = (x - \alpha) H_1(x)
\end{equation}
uniquely where $\alpha \in \mbb Z$ and unique monic irreducible polynomial $H_1(x) \in \mbb Z[x]$.
\end{lem}

\begin{proof}
Let $s \leq 92$. Assume that $G_1(x)$ is reducible. Then we derive from \lemref{lem2k_le_n_le_4k} that either $(n, k, s)$ belongs to the finite set stated in \lemref{lem2k_le_n_le_4k} or $G_1(x)$ has no factor of degree $k \geq 2$. Now the assertion follows immediately.
\end{proof}

\section{Irreducibility of $G_1(x, 2, s)$ for $s \in \{ 3, 7, 15 \}$}

We compute 
\begin{equation}\label{eq_new_G2}
G_1(x) = b_2 x^2 - 2(1+s) b_1 x  + b_0 \frac{(2+s) (1+s)}{2}
\end{equation}
where $|b_0| = |b_2| = 1$. For the irreducibility of $G_1(x)$ it suffices to show that the polynomials 
$$
x^2 \pm 2(1+s) b_1 x \pm \frac{(2+s) (1+s)}{2}
$$
are irreducible. We prove

\begin{lem}\label{lem_new_5}
The polynomial \eqref{eq_new_G2} with $s = 3$ and $s = 15$ are irreducible for every $b_1 \in \mbb Z$. Also the polynomial \eqref{eq_new_G2} with $s = 7$ is irreducible for every $b_1 \in \mbb Z$ except for $b_1 = 0$ where the polynomial is $x^2 - 36$.
\end{lem}

\begin{proof}
The proof depends on a well known assertion that a quadratic polynomial is irreducible if and only if its discriminant is not a square. We consider $x^2 - 8b_1 x +10$ obtained from \eqref{eq_new_G2} by putting $b_0 = 1 = b_2$. Suppose it is reducible. Then its discriminant $ (8 b_1)^2 - 40 = m^2$ for an integer $m \geq 0$. Thus $ ( 8 b_1 - m, 8b_1 +m) \in \{ ( 1, 40), ( 2, 20), (4, 10), (5, 8) \}$ and then $ 16 b_1 \in \{ 41, 22, 14, 13 \}$. This is not  possible since none of $41, 22, 14, 13$ is divisible by $16$. The assertion follows similarly for all other cases.
\end{proof}

\section{ $G_1(x)$ divisible by a linear factor}

For $s \leq 92$, we see from \lemref{lem_s_92} except for finitely many cases, $G_1(x)$ is either irreducible or divisible by a linear factor. In this section, we consider the case where $G_1(x)$ is divisible by a linear factor. Then we derive from \lemref{lem2k_le_n_le_4k} that $n$ is bounded by a computable number depending only on $s$. If $s$ is restricted to $30$, we prove a more precise assertion.

\begin{thm}\label{thm_new_1}
Let $n \geq 2, s \leq 30$ and $G_1(x, 2, 7) \neq x^2 - 36$. Assume that $G_1(x) = G_1( x, n, s)$ is divisible by a linear factor and
\begin{equation}\label{eq_thm1}
(n, k, s) \not\in \{ (4, 2, 7), (4, 2, 23), ( 9, 2, 19), ( 16, 2, 14), (16, 3, 19), ( 10, 5, 4) \}.
\end{equation}
Then $(n, s) \in X$ where 
\begin{align*}
X = & \{ (6, 3), ( 4, 5), ( 8, 11), ( 72, 11), ( 3, 15), ( 10, 15), ( 4, 15), ( 12, 15), ( 8, 15),  (16, 17), \\ &  ( 272, 17), ( 8, 27), ( 16, 29), (786600, 25),  (786600, 26)   \}.
\end{align*} 
\end{thm}

\begin{proof}
By definition, the assumption \eqref{eq_thm1} is interpreted as $G_1(x)$ has no factor of degree $2$ at $(n, s) \in \{ (4, 7), (4, 23), ( 9, 19), (16, 14) \}$, no factor of degree $3$ at $(n, s) = (16, 19)$ and no factor of degree $5$ at $(n, s) = (10, 4)$. Assume that $G_1(x)$ is divisible by a linear factor. Then, as in \lemref{FFL} of \cite{FFL}, we have
\begin{equation}\label{eq_neq_1}
n = \prod\limits_{p | n} p^{\nu_p(n)} = \prod\limits_{p \leq s} p^{\nu_p(n)}
\end{equation}
where
\begin{equation}\label{eq_new_2}
p^{\nu_p(n)} \quad \text{for} \quad p \leq s
\end{equation}
and
\begin{equation}\label{eq_new_3}
p \mid \frac{(n+1) \dots (n+s)}{s!} \quad \text{for} \quad p \mid n.
\end{equation}
Denote by $T$ the set of all pairs $(n,s) $ satisfying \eqref{eq_neq_1}, \eqref{eq_new_2} and \eqref{eq_new_3}.  By applying \lemref{FFL} with $ l = 0, k = 1$ to all pairs $(n, s) \in T$, we check that \lemref{FFL} does not hold for the following set $T_1$ of pairs $(n, s)$ given by
\begin{align*}
\{& (2, 3), (6, 3), (4, 5), (2, 7), (4, 7), (8, 11), ( 72, 11), ( 8, 13), ( 3, 15), ( 2, 15), (10, 15), (4, 15), (12, 15), \\ & (8, 15), (16, 17), (272, 17), ( 16, 19), ( 6, 23), ( 4, 23), (16, 23), (16, 24), ( 16, 26), ( 8, 27), ( 216, 29), \\ & ( 16, 19), ( 786600, 25), ( 786600, 26) \}.
\end{align*}
Denote by $T_2$ the pairs $(n, s)$ with $n =2$. These are excluded by \lemref{lem_new_5}. Denote by $T_3$ the complement of $T_2 \cup \{ (3, 15) \}$ in $T_1$. Then all the pairs $(n, s) \in T_1$ satisfies $n \geq 4$. Therefore we derive \eqref{eq_new_G1} uniquely for every $(n, s) \in T_3$ by \lemref{lem_s_92}. Denote by $T_4$ the set obtained by applying \lemref{FFL} with $l =1$ and $k = [\frac{n}{2}]$ to $G_1(x)$ with $(n, s) \in T_3$. We calculate $T_4 = X \setminus \{ (3, 15) \}$. Now the assertion of \thmref{thm_new_1} follows immediately.
\end{proof}

Now we give an application of \thmref{thm_new_1} with $G_1(x)$ replaced by $g_1(x)$. We prove

\begin{cor} \label{cor_new_1}
Let $s \leq 30$. If $g_1(x)$ is reducible, the 
$$
(n, s) \in \{ (786600, 25), (786600, 26) \}.
$$
\end{cor}

This implies that $g_1(x)$ with $s \leq  24$ is irreducible which includes a new proof of a result of Nair and Shorey \cite{NaSh1}. We refer to \cite{ShSi} for a complete account of results proved on the irreducibility of $g_1(x)$. The results of Hajir and its refinement by Nair and Shorey and Jindal, Laishram and Sarma  depend on algebraic results of Hajir \cite{Haj1} on the Newton polygons. Our proof of \corref{cor_new_1} is new in the sense that it does not use the above results of Hajir \cite{Haj1} on Newton polygons.

\begin{proof}[Proof of Corollary]
Let $s \leq 30$ and $G_1(x) = g_1(x)$ be reducible. We compute that the values of $g_1(x)$ at $(n, s) \in \{ ( 4, 7), ( 4, 23), ( 9, 19), ( 16, 14), ( 16, 19), ( 10, 4) \}$ are irreducible. Now we derive from \lemref{lem2k_le_n_le_4k} that $g_1(x)$ is divisible by a linear factor and we check  $g_1(x)$ at $(n, s) = (2, 7)$ is irreducible. Therefore the assumption of \thmref{thm_new_1} with $G_1$ replaced by $g_1$ are satisfied. Hence we conclude $(n, s) \in X$ by \thmref{thm_new_1}. Now we compute $g_1(x)$ with $(n, s) \in X$ are irreducible. This is a contradiction since $g_1(x)$ is divisible by a linear factor.
\end{proof}

\end{document}